\documentclass[12pt]{amsart}
\usepackage{amssymb}

\usepackage{graphicx}
\usepackage{float}
\usepackage{lmodern}
\usepackage{graphicx}

\theoremstyle{plain}

\newtheorem{corollary}{Corollary}

\newtheorem{lemma}{Lemma}

\newtheorem{proposition}{Proposition}
\newtheorem{remark}{Remark}

\numberwithin{equation}{section}

\begin{document}
\title[The singular values of the logarithmic potential transform ]{The singular values of the logarithmic potential transform on bound states spaces of Landau Hamiltonian }

\author[M. El-Omari]{M.El-Omari}

\begin{abstract}
The singular values of the logarithmic potential transform on the generalized Bergmann space is calculated explicitly, too behavior in infinity. %Our proof does not require any ``smoothing'' of the M5-brane, unlike previous analyses of the cancellation of local anomalies. 
\end{abstract}

\maketitle

\section{Introduction}

Let $\mathbb{D}$ be the complex unit disk endowed with its Lebesgue measure $%
\mu $ and let $\partial \mathbb{D}$ be its boundary. Denote by $L^{2}(%
\mathbb{D},d\mu )$ the space of complex-valued measurable functions which are
$d\mu $ square integrable on $\mathbb{D}$. The logarithmic potential transform $:L^{2}(%
\mathbb{D})\rightarrow L^{2}(\mathbb{D})$ is defined by
\begin{equation}
\mathcal{L}[f](z)=-\frac{1}{\pi }\int_{\mathbb{D}}\frac{f\left( \xi \right)
}{\xi -z}log\left(\frac{1}{|z-\xi|}\right)d\mu \left( \xi \right)  \tag{1.1}
\end{equation}
This operator is very important as the transformed Cauchy and it often appears in Analysis \cite{L}%
.\\ The dimensional analysis \cite{B,L} and scaling arguments \cite{D2002} form an integral part in
theoretical physics to solve some important problems without doing much calculation.\\
The logarithmic potential in physics forms an interesting one as it provides some unusual
prediction about the system. Moreover, this potential can be used suitably to illustrate
some of the important features of field theory such as dimensional regularization and
renormalization. In most of our text books, this potential is not discussed at detail;
although the calculations are quite simple to demonstrate some of its unique features. We
have obtained the bound state energy of this logarithmic potential through uncertainty
principle, phase space quantization and Hellmann-Feynman theorem.\\
In \cite{AK} the authors have been dealing with the restriction of
$\mathcal{L}$ to the space $L_{a}^{2}(\mathbb{D})$ of analytic $\mu $-square
integrable on $\mathbb{D}$.\ \ They precisely have considered the projection
operator $P_{0}$: $L^{2}(\mathbb{D})\rightarrow L_{a}^{2}(\mathbb{D})$ and
they have proved that the singular values $\lambda _{k}$ of $\mathcal{L}%
P_{0} $, which turn out to be eigenvalues of the operator $\sqrt{(\mathcal{LP%
}_{0})^{\ast }(\mathcal{L}P_{0})}$) behave like $k^{-1}$ as $k$ goes to $%
\infty $.\ \ They also concluded that $\mathcal{L}P_{0}$ belongs to the
Schatten class $S_{1,\infty }$.\newline
Now, consider the following \textit{weighted} logarithmic potential transform
\begin{equation}
\mathcal{L}_{\sigma }[f](z)=-\frac{1}{\pi }\int_{\mathbb{D}}\frac{f\left(
\xi \right) }{\xi -z}log\left(\frac{1}{|z-\xi|}\right)\left( 1-\xi \overline{\xi }\right) ^{^{\sigma -2}}d\mu
\left( \xi \right) 
\end{equation}
defined on the space $L^{2,\sigma }(\mathbb{D})$ of complex-valued
measurable functions which are $(1-\xi \overline{\xi })^{\sigma -2}d\mu (\xi
)$-square integrable on $\mathbb{D}$ \ where $\sigma >1$\ is \ a fixed
parameter.\ We observe that the subspace $L_{a}^{2,\sigma }(\mathbb{D})$ of
analytic functions on $\mathbb{D}$ and belonging to $L^{2,\sigma }(\mathbb{D}%
)$ coincides with the eigenspace
\begin{equation}
\mathcal{A}_{0}^{\sigma }\left( \mathbb{D}\right) :=\left\{ \psi \in
L^{2,\sigma }(\mathbb{D}),\Delta _{\sigma }\psi =0\right\}  \tag{1.3}
\end{equation}
of the second order differential operator
\begin{equation}
\Delta _{\sigma }:=-4\left( 1-z\overline{z}\right) \left( \left( 1-z%
\overline{z}\right) \frac{\partial ^{2}}{\partial z\partial \overline{z}}%
-\sigma \overline{z}\frac{\partial }{\partial \overline{z}}\right)  \tag{1.4}
\end{equation}
\ known as the $\sigma $-weight Maass Laplacian and its discrete eingenvalues are given by
\begin{equation*}
\epsilon _{m}:=4m\left( \sigma -1-m\right) ,\ m=0,1,2,...,\left\lfloor
\left( \sigma -1\right) /2\right\rfloor
\end{equation*}
with their corresponding eigenspaces
\begin{equation}
\mathcal{A}_{m}^{\sigma }(\mathbb{D}):=\left\{ \psi \in L^{2,\sigma }\left(
\mathbb{D}\right) \text{ and }\Delta _{\sigma }\psi =\epsilon _{m}^{\sigma
}\psi \right\}  \tag{1.6}
\end{equation}
are here called \textit{generalized Bergman} spaces since... \

After noticing that, we here deal with analogous questions as in \cite{AK}%
 in the context of the weighted Cauchy transform $(1.2)$ and for
its restriction to the space $\mathcal{A}_{m}^{\sigma }(\mathbb{D})$.\ That
is , we are concerned with the operator $\mathcal{C}_{\sigma }P_{m}^{\sigma
} $ where $P_{m}^{\sigma }$ is the projection $L^{2,\sigma }(\mathbb{D}%
)\rightarrow \mathcal{A}_{m}^{\sigma }(\mathbb{D})$.\ \ The obtained results
are as follows.

Firstly , we find that the singular values of $\mathcal{L}_{\sigma
}P_{m}^{\sigma }$.\\

For $k\neq m$, can be expressed as
 $$\lambda_k=\sqrt{J_1+J_2+J_3}$$
where
$$J_1=\left(\frac{(1+k-m)_m}{m!(k-m+1)}\right)^2\sum_{n=0}^{\infty}{A_n\frac{\Gamma(2n+2k-2m+6-1)\Gamma(4\nu-2m-1)}{\Gamma(2n+2k-4m+4\nu+6)}}$$

$$J_2=\left(\frac{\alpha^{\nu,m}_k}{2\nu-m-1}\right)^2\sum_{n=0}^{\infty}{A_n\frac{\Gamma(4\nu-2m-1)\Gamma(2n+2)}{\Gamma(2n+4\nu-2m+1)}}$$

and 

$$ J_3=\frac{(1+k-m)_m \alpha^{\nu,m}_k}{m!(k-m+1)(2\nu-m-1)}\left(\sum_{n=0}^{\infty}{A_n\frac{\Gamma(k-m+2)\Gamma(4\nu-2m-1)}{\Gamma(4\nu-k-3m)}}\right)$$

For $k=m$ can be expressed as
\begin{equation}
\lambda^2_k=\frac{\alpha^{\nu,m}_k \left(2(2\nu-m)-1\right)}{8\left(\pi(2\nu-m+1)\right)}\sum_{n=0}^{\infty}{\frac{B_n}{n+2\nu-m}}
\end{equation}
where 
$$B_n=\sum_{n=0}^{\infty}{\frac{\Gamma(-m+1)\Gamma(2\nu-m)\Gamma(2(\nu-m)+1)}{n!\Gamma(2(\nu-m)\Gamma(2\nu-m+2)} }$$

$$\alpha^{\nu,m}_k = \frac{\Gamma(2)\Gamma(2(m-\nu)+1)}{\Gamma(m+1)\Gamma(2+m-2\nu)}$$
Secondly, we show that these singular values behave like
$$\lambda_k\sim C \sqrt{k^{m-4\nu+1}},\ \text{as}\ k\longrightarrow\infty$$
where $C$ is a constant.\\
The paper is organized as follows.\ \ In section 2, we review the definition
of the weighted logarithmic potential transform as well as some of its needed properties.\
\ Section 3 deals with some basic facts on the spectral theory of \ Maass
Laplacians on the Poincar\'{e} disk.\ \ In Section 4, a precise description
of the generalized Bergmann spaces is reviewed.\ \ Section 5 is devoted to
the computation of the singular values of the weighted logarithmic potential transform\\
The asymptotic behavior of theses singular values is established in Section 6.

\section{The weighted Logarithmic Potential transform $\mathcal L_{\nu}$}
\subsection{The case $\nu=1$}
Let $\mathbb{D}$ the complex unit disk endowed with its Lebesgue measure $\mu$ and let $\partial\mathbb{D}$  its boundary denote by $L^2(\mathbb{D})$ the space of complex-valued measurable functions on $\mathbb{D}$ with finite norm
\begin{equation*}
			\left\|f\right\| = \int_{\mathbb{D}}{|f(\xi)|^2d\mu(\xi)}  \tag{1.1}
	\end{equation*}
The Logarithmic Potential operator $\mathcal L:L^2(\mathbb{D})\rightarrow L^2(\mathbb{D})$ is defined by
\begin{equation}
	\mathcal L[f](z)=\int_{\mathbb{D}}{f(\xi)\log\left(\frac{1}{\left|\xi-z\right|}\right)d\mu(\xi)} \ \ \ \ \ \ \ \   \tag{1.2}
\end{equation}

\subsection{The case of $\nu\geq 1$}
We fix a real parameter $\nu$ such that $2\nu > 1$ and we consider the following weighted Logarithmic Potential transform
\begin{equation}
	\mathcal L_{\nu}[f](z)=\int_{\mathbb{D}}{f(\xi)\log\left(\frac{1}{\left|\xi-z\right|}\right)(1-\xi\overline{\xi})^{2\nu-2}d\mu(\xi)} \ \ \ \ \ \ \ \   \tag{1.4}
\end{equation}
defined on the space $L^{2,\nu}(\mathbb{D})$ complex-valued measurable functions which are $(1-\xi\overline{\xi})^{2\nu-2}d\mu(\xi)$-square integrable  on $\mathbb{D}$. As a convolution of $L^{2,\nu}$-functions with the compactly supported measure $\frac{(1-\xi)^{2\nu-2}}{\xi} \mathbb{1}_{\mathbb{D}}d\mu(\xi)$  $\mathcal L_{\nu} :L^{2,\nu}(\mathbb{D}) \rightarrow L^{2,\nu}(\mathbb{D})$ is obviously bounded. Moreover, it is not hard to show that $\mathcal L_{\nu}$ is in fact compact \cite{AKL1990}. This raises a question concerning the spectral picture of $\mathcal L_{\nu}$.

\section{The Landau Hamiltonian $H_{\protect\nu }$ on the Poincar\'{e} disk $%
\mathbb{D}$}

Let $\mathbb{D}$ =$\left\{ z\in \mathbb{C},z\overline{z}<1\right\} $ be the
complex unit disk with the Poincar\'{e} metric $ds^{2}=4\left( 1-z\overline{z%
}\right) ^{-2}dzd\overline{z}.$ $\mathbb{D}$ is a complete Riemannian 
manifold with all sectional curvature equal $-1.$ It has an ideal boundary $%
\partial \mathbb{D}$ identified with the circle $\left\{ \omega \in \mathbb{C%
},\omega \overline{\omega }=1\right\} .$ One refers to points $\omega \in
\partial \mathbb{D}$ as points at infinity. The geodesic distance between
two points $z$ and $w$ is given by%
\begin{equation}
\cosh d\left( z,w\right) =1+\frac{2\left( z-w\right) \left( \overline{z}-%
\overline{w}\right) }{\left( 1-z\overline{z}\right) \left( 1-w\overline{w}%
\right) }  \tag{2.1}
\end{equation}%
By \cite{FV} the Schr\"{o}dinger operator on $\mathbb{D}$ with
constant magnetic field of strength proportional to $\nu >0$ can be written
as $:$%
\begin{equation}
\mathcal{L}_{\nu }:=-(1-|z|^{2})^{2}\frac{\partial ^{2}}{\partial z\partial
\overline{z}}-\nu z\left( 1-\left\vert z\right\vert ^{2}\right) \frac{%
\partial }{\partial z}+\nu \overline{z}\left( 1-|z|^{2}\right) \frac{%
\partial }{\partial \overline{z}}+\nu ^{2}|z|^{2}.  \tag{2.2}
\end{equation}%
which is also called Maass Laplacian on the disk. A
slight modification of $\mathcal{L}_{\nu }$ is given by the operator
\begin{equation}
H_{\nu }:=4\mathcal{L}_{\nu }-4\nu ^{2}  \tag{2.3}
\end{equation}%
acting in the Hilbert space
\begin{equation}
L^{2,0}(\mathbb{D}):=\left\{ \varphi :\mathbb{D\rightarrow C},\int_{\mathbb{D%
}}\left\vert \varphi (z)\right\vert ^{2}\left( 1-\left\vert z\right\vert
^{2}\right) ^{-2}d\mu (z)<+\infty \right\} ,  \tag{2.4}
\end{equation}%
For our purpose, we shall consider the unitary equivalent realization $%
\widetilde{H}_{\nu }$ of the operator $H_{\nu }$ in the Hilbert space
\begin{equation}
L^{2,\nu }(\mathbb{D}):=\left\{ \varphi :\mathbb{D\rightarrow C},\int_{%
\mathbb{D}}\left\vert \varphi (z)\right\vert ^{2}\left( 1-|z|^{2}\right)
^{2\nu -2}d\mu (z)<+\infty \right\} ,  \tag{2.6}
\end{equation}%
which is defined by
\begin{equation}
\mathsf{\ }\widetilde{H}_{\nu }:=\mathfrak{Q}_{\nu }^{-1}H_{\nu }\mathfrak{Q}%
_{\nu },  \tag{2.7}
\end{equation}%
where $\mathfrak{Q}_{\nu }:L^{2,\nu }(\mathbb{D})\rightarrow L^{2,0}(\mathbb{%
D})$ is the unitary transformation defined by the map $\varphi \mapsto
\mathfrak{Q}_{\nu }\left[ \varphi \right] :=(1-|z|^{2})^{-\nu }\varphi .$ \
\ Different aspects of the spectral analysis of the operator $\widetilde{H}%
_{\nu }$ have been studied by many authors. For instance, note that $\widetilde{H}_{\nu }$
is an elliptic densely defined operator on the Hilbert space $L^{2,\nu
}\left( \mathbb{D}\right) $ and admits a unique self-adjoint realization
that we denote also by $\widetilde{H}_{\nu }.$ The spectrum\ of $\widetilde{H%
}_{\nu }$ in $L^{2,\nu }(\mathbb{D})$ consists of two parts: $(i)$\textit{\ }%
a continuous part $\left[ 1,+\infty \right[ $ , which correspond to \textit{%
scattering states}, $(ii)$ a finite number of eigenvalues (\textit{%
hyperbolic Landau levels}) of the form %
\begin{equation}
\epsilon _{m}^{\nu }:=4(\nu -m)\left( 1-\nu +m\right) ,m=0,1,2,\cdots ,\left[
\nu -\frac{1}{2}\right]   \tag{2.6}
\end{equation}%
with infinite degeneracy, provided that $2\nu >1$.\ \ To the eigenvalues in $\left( 2.6\right) $ correspond eigenfunctions which
are called \textit{bound states }since the particle in such a state cannot
leave the system without additionnal energy. A concrete description of these
bound states spaces will be the goal of the next section.

\section{The bound states spaces $\mathcal{A}_{\protect\nu ,m}^{2}(\mathbb{D})$}

Here, we consider the eigenspace
\begin{equation}
\mathcal{A}_{\nu ,m}^{2}(\mathbb{D}):=\left\{ \Phi :\mathbb{D}\mathbf{%
\rightarrow }\mathbb{C},\Phi \in L^{2,\nu }\left( \mathbb{D}\right) \text{
and }\widetilde{H}_{\nu }\Phi =\epsilon _{m}^{\nu }\Phi \right\}   \tag{3.1}
\end{equation}% 
 See \cite{MO,W}, for the following proposition.
\begin{proposition} Let  $2\nu >1$ and $m=0,1,2,\cdots ,\left[ \nu -%
\frac{1}{2}\right] .$\textit{Then, we have}

$\left( i\right) $ \textit{an orthogonal basis} of \ $\mathcal{A}_{\nu
,m}^{2}(\mathbb{D})$ \textit{is given by the functions}

\begin{equation}
\phi _{k}^{\nu ,m}(z):=|z|^{|m-k|}(1-|z|^{2})^{-m}e^{-i\left( m-k\right)
\arg z}  \notag
\end{equation}%
\begin{equation*}
\times _{2}F_{1}\left( -m+\frac{m-k+|m-k|}{2},2\nu -m+\frac{|m-k|-m+k}{2}%
,1+|m-k|;|z|^{2}\right)
\end{equation*}%
$k=0,1,2,\cdots ,$\textit{\ in terms of a terminating }$_{2}F_{1}$\textit{%
Gauss hypergeometric function. }\\
$\left( ii\right) $ \textit{the norm square\ of  }$\phi _{k}^{\nu ,m}$%
\textit{\ in }$L^{2,\nu }(\mathbb{D})$\textit{\ is given by}
\begin{equation*}
\left\Vert \phi _{k}^{\nu ,m}\right\Vert ^{2}=\frac{\pi \left( \Gamma \left(
1+\left\vert m-k\right\vert \right) \right) ^{2}}{\left( 2(\nu -m)-1\right) }%
\frac{\Gamma \left( m-\frac{|m-k|+m-k}{2}+1\right) \Gamma \left( 2\nu -m-%
\frac{|m-k|+m-k}{2}\right) }{\Gamma \left( m+\frac{|m-k|-m+k}{2}+1\right)
\Gamma \left( 2\nu -m+\frac{|m-k|-m+k}{2}\right) }.
\end{equation*}
\end{proposition}

\bigskip
\begin{corollary} 
	The functions $\left\{ \Phi _{k}^{\nu,m}\right\}$, $k=0,1,2,\ldots$ given by 
	\begin{equation*}
		\Phi_{k}^{\nu ,m} \left( z\right) :=\left( -1\right)^{k} 
		\left( \frac{2\left( \nu -m\right) -1}{\pi } \right)^{\frac{1}{2}} 
		\left( \frac{k!\Gamma \left( 2\left( \nu -m\right) + m\right) }{m!\Gamma \left( 2\left( \nu-m\right) +k\right) }\right)^{\frac{1}{2}}
	\end{equation*}
	\begin{equation*}
		\times \left( 1-\left\vert z\right\vert^{2}\right)^{-m}\overline{z}^{m-k}P_{k}^{\left( m-k,2\left( \nu -m\right) -1\right) }
		\left( 1-2z\overline{z}\right)
	\end{equation*}
	in terms of Jacobi polynomials constitute an orthonormal basis of $\mathcal{A}_{m}^{2,\nu }\left( \mathbb{D}\right)$
\end{corollary}
\textbf{Proof. } Write the connection between the\textbf{\ }$_{2}F_{1}$-sum
and the\textbf{\ }Jacobi polynomial
$$ P_{k}^{\alpha ,\beta }(u)=\frac{( 1+\alpha)_k}{k!}.\ _{2}F_{1}( -k,1+\alpha +\beta +k,1+\alpha;\frac{1-u}{2})$$
then the functions
\begin{equation*}
\phi _{k}^{\nu ,m}\left( z\right) =\frac{\left( -1\right) ^{\min \left(
m,k\right) }}{\left( 1-\left\vert z\right\vert ^{2}\right) ^{m}}\left\vert
z\right\vert ^{\left\vert m-k\right\vert }e^{-i\left( m-k\right) \arg
z}P_{\min \left( m,k\right) }^{\left( \left\vert m-k\right\vert ,2\left( \nu
-m\right) -1\right) }\left( 1-2z\overline{z}\right)
\end{equation*}%
constitute an orthonormal basis of $\mathcal{A}_{\nu ,m}^{2}.$The norm
square of $\phi _{k}^{\nu ,m}$ in $L^{2,\nu }\left( \mathbb{D}\right) $ is
given by
\begin{equation*}
\left\Vert \phi _{k}^{\nu ,m}\right\Vert ^{2}=\frac{\pi }{\left( 2\left( \nu
-m\right) -1\right) }\frac{\left( m\vee k\right) !\Gamma \left( 2\left( \nu
-m\right) +m\wedge k \right) }{\left( m\wedge k\right) !\Gamma \left( 2\left(
\nu -m\right) +m\vee k\right) }.
\end{equation*}%
Here, $m\wedge k:=\min \left( m,k\right) $ and $m\vee k:=\max \left(m,k\right)$. Thus, the set of functions
\begin{equation*}
\Phi _{k}^{\nu ,m}:=\frac{\phi _{k}^{\nu ,m}}{\left\Vert \phi _{k}^{\nu,m}\right\Vert }, \quad k=0,1,2,...
\end{equation*}
is an orthonormal basis of $\mathcal{A}_{\nu ,m}^{2}\left( \mathbb{D}\right)$ and can be rewritten as
\begin{equation}
	\Phi _{k}^{\nu ,m}\left( z\right) = \left( -1\right)^{k} \left( \frac{2( \nu -m) -1}{\pi } \right)^{\frac{1}{2}}
	\left( \frac{k!\Gamma\left( 2( \nu -m) +m\right) }{m!\Gamma \left( 2( \nu-m) +k\right) }\right)^{\frac{1}{2}} \tag{$\star$}
\end{equation}
\begin{equation*}
	\times \left( 1-\left\vert z\right\vert ^{2}\right)^{-m} 
	\overline{z}^{m-k}P_{k}^{\left( m-k,2( \nu -m) -1\right) }\left( 1-2z\overline{z}\right)
\end{equation*}
by making appeal to the identity \cite[p. 63]{S} 
\begin{equation*}
\frac{\Gamma \left( m+1\right) }{\Gamma \left( m-s+1\right) }P_{m}^{\left(
-s,\alpha \right) }\left( u\right) =\frac{\Gamma \left( m+\alpha +1\right) }{%
\Gamma \left( m-s+\alpha +1\right) }\left( \frac{u-1}{2}\right)
^{s}P_{m-s}^{\left( s,\alpha \right) }\left( u\right) ,1\leq s\leq m
\end{equation*}\label{s}%

for $s=m-k,$ $t=1-2\left\vert z\right\vert ^{2}$ and $\alpha =2\left( \nu
-m\right) -1....\square $\\
\begin{corollary} \textit{The} $L^{2}-$\textit{eigenspace} $\mathcal{A}%
_{\nu ,0}^{2}\left( \mathbb{D}\right) $, \textit{corresponding to }$m=0$%
\textit{\ in }$\left( 3.1\right) $\textit{\ and associated to the bottom
energy }$\epsilon _{0}^{\nu }=0$\textit{\ in }$\left( 2.6\right) ,$\textit{\
reduces further to the weighted Bergman space consisting of holomorphic
functions }$\phi $\textit{\ }: $\mathbb{D\rightarrow C}$ \textit{such that}
\begin{equation*}
\int\limits_{\mathbb{D}}\left\vert \phi \left( z\right) \right\vert
^{2}\left( 1-\left\vert z\right\vert ^{2}\right) ^{2\nu -2}d\mu \left(
z\right) <+\infty .
\end{equation*}%
\end{corollary}

\section{Computation of the singular values $\lambda_k$ }
Elements of this basis are given in terms of Jacobi polynomials as\\ 
\begin{equation}
\phi _{k}^{\nu ,m}\left( z\right) =\frac{\left( -1\right) ^{\min \left(
m,k\right) }}{\left( 1-\left| z\right| ^{2}\right) ^{m}}\left| z\right|
^{\left| m-k\right| }e^{-i\left( m-k\right) \arg z}P_{\min \left( m,k\right)
}^{\left( \left| m-k\right| ,2\left( \nu -m\right) -1\right) }\left( 1-2z%
\overline{z}\right)  \tag{5.1}
\end{equation}
The norm square of $\phi _{k}^{\nu ,m}$ in $L^{2,\nu }\left( \mathbb{D}%
\right) $ is given by
\begin{equation}
\rho _{k}^{\nu ,m}=\frac{\pi }{\left( 2\left( \nu -m\right) -1\right) }\frac{%
\left( m\vee k\right) !\Gamma \left( 2\left( \nu -m\right) +m\wedge k \right)
}{\left( m\wedge k\right) !\Gamma \left( 2\left( \nu -m\right) +m\vee
k\right) }.  \tag{5.2}
\end{equation}
Here, $m\wedge k:=\min \left( m,k\right) $ and $m\vee k:=\max \left(
m,k\right) .$ Let us introduce the notationThe set of functions
\begin{equation}
\gamma _{k}^{\nu ,m}:=\frac{\left( -1\right) ^{m\wedge k}}{\sqrt{\rho
_{k}^{\nu ,m}}},k=0,1,2,...  \tag{5.3}
\end{equation}
So that we consider the elements
\begin{equation}
\Phi _{k}^{\nu ,m}\left( z\right) :=\gamma _{k}^{\nu ,m}\frac{1}{\left( 1-z%
\overline{z}\right) ^{m}}\left| z\right| ^{\left| m-k\right| }e^{-i\left(
m-k\right) \arg z}P_{\min \left( m,k\right) }^{\left( \left| m-k\right|
,2\left( \nu -m\right) -1\right) }\left( 1-2z\overline{z}\right)  \tag{4}
\end{equation}
\subsection{The action $\mathcal L_{\nu}$}
\begin{lemma}
We set $z=\rho e^{it}$, and $I=-\int_{0}^{2\pi}{e^{i(k-m)\theta}log\left(|z-re^{i\theta}|\right)\frac{d\theta}{2\pi}}$, we have
\begin{equation}
	\begin{cases}
    I=-\log\left(\rho\wedge r\right) &	 k=m\\
		I=\frac{e^{i(k-m)t}}{2|m-k|}\left(\left(\frac{r}{\rho}\right)^{m-k}\wedge \left(\frac{r}{\rho}\right)^{m-k}\right) & k\neq m
  \end{cases}
\end{equation}	
\end{lemma}
\begin{proof}[\textbf{Proof}] By \cite{AK}, it remain to prove that this lemma for $k<m$.\\
We have 
\begin{equation*}
\int_{0}^{2\pi}{e^{i(k-m)\theta}log\left(\left|\rho e^{it}-r e^{i\theta}\right|\right)d\theta}%
=-\int_{0}^{2\pi}{e^{i(m-k)(-\theta)}log\left(\left|r e^{i(-t)}-\rho e^{i(-\theta)}\right|\right)d(-\theta)}%
\end{equation*}
The function $\theta\rightarrow e^{i(m-k)(-\theta)}log\left(\left|r e^{i(-t)}-\rho e^{i(-\theta)}\right|\right)$ is a periodic mapping with the period equal $2\pi$, then 
$$\int_{0}^{2\pi}{e^{i(k-m)\theta}log\left(\left|\rho e^{it}-r e^{i\theta}\right|\right)d\theta}$$
$$=-\int_{0}^{2\pi}{e^{i(m-k)(-\theta)}log\left(\left|r e^{i(-t)}-\rho e^{i(-\theta)}\right|\right)d(-\theta)}$$
$$=\frac{e^{i(k-m)t}}{2(m-k)}\times \left(\left(\frac{r}{\rho}\right)^{m-k}\wedge\left(\frac{r}{\rho}\right)^{m-k}\right)$$
\end{proof}

\begin{lemma}
For all $\lambda\in\partial D$. $\mathcal L_{\nu}$ commutes with the rotations $R_{\lambda}$, where
$$\left(R_{\lambda}f\right)(z)=f\left(\lambda z\right)$$
\end{lemma}
\begin{proof}[\textbf{Proof}]
We observe that 
$$\left(R_{\lambda}\phi^{\nu,m}_k\right)(z)=\lambda^{k-m}\phi^{\nu,m}_k (z),\ \forall k\neq m$$
\end{proof}

\begin{corollary}
$\left\{\mathcal L_{\nu}\left(\phi^{\nu,m}_k\right)\right\}^{\infty}_{k=0}$ are orthonormal in $L^{2,\nu}\left(\mathbb{D}\right)$
\end{corollary} 
\begin{proof}[\textbf{Proof}]
As $R_{\lambda}$ is an isometry of $L^{2,\nu}(\mathbb{D})$,
$$\left(\mathcal L_{\nu}\left(\phi^{\nu,m}_k\right),\mathcal L_{\nu}\left(\phi^{\nu,m}_j\right)\right)$$
$$=\left(R_{\lambda}\mathcal L_{\nu}\left(\phi^{\nu,m}_k\right),R_{\lambda}\mathcal L_{\nu}\left(\phi^{\nu,m}_j\right)\right)$$
$$=\left(\mathcal L_{\nu}R_{\lambda}\left(\phi^{\nu,m}_k\right),\mathcal L_{\nu}R_{\lambda}\left(\phi^{\nu,m}_j\right)\right)$$
$$=\overline{\lambda^{j-k}}\left(\mathcal L_{\nu}\left(\phi^{\nu,m}_k\right),\mathcal L_{\nu}\left(\phi^{\nu,m}_j\right)\right),\ if\ m>k$$
or
$$=\lambda^{k-j}\left(\mathcal L_{\nu}\left(\phi^{\nu,m}_k\right),\mathcal L_{\nu}\left(\phi^{\nu,m}_j\right)\right),\ if\ m<k$$
For all $\lambda\in\partial D$, since $\lambda\neq 0$, we have
$$\left(\mathcal L_{\nu}\left(\phi^{\nu,m}_k\right),\mathcal L_{\nu}\left(\phi^{\nu,m}_j\right)\right)=0\ if\ j\neq k$$
\end{proof}

%\begin{lemma}
%If we denote $^1\phi^{\nu,m}_k(z),\ \text{if}\ k>m$ and $^2\phi^{\nu,m}_k(z),\ \text{if}\ k<m$, we have 
%$$\mathcal L_{\nu}\left(^1\phi^{\nu,m}_k\right)(z)=\frac{\Gamma(k+1)\Gamma(2\nu-m)}{\Gamma(m+1)\Gamma(2\nu-k)}\mathcal L_{\nu}\left(^2\phi^{\nu,m}_k\right)(z)$$
%\end{lemma}
%\begin{proof}[\textbf{Proof}]
%Just use
%\begin{equation*}
%\frac{\Gamma \left( m+1\right) }{\Gamma \left( m-s+1\right) }P_{m}^{\left(
%-s,\alpha \right) }\left( u\right) =\frac{\Gamma \left( m+\alpha +1\right) }{%
%\Gamma \left( m-s+\alpha +1\right) }\left( \frac{u-1}{2}\right)
%^{s}P_{m-s}^{\left( s,\alpha \right) }\left( u\right) ,1\leq s\leq m
%\end{equation*}%
%\end{proof}

\begin{proposition}
The action of the operator $\mathcal L$ on a basis element $\phi^{\nu,m}_k$, is of the form:\\
If $k=m$, We put $z=\rho e^{i\theta}$ then

\begin{equation*}
\mathcal L_{\nu}(\phi^{\nu,m}_k)(z) = \beta\left(\rho\right)\;
 _3F_2\left(\begin{array}{cc} -m+1,2\nu-m,2\nu-m+1\\ 2(\nu-m), 2\nu-m+2\end{array}\mid 1-\rho^2\right)
\end{equation*}
with 
$\displaystyle \beta\left(\rho\right) = \frac{\alpha^{\nu,m}_k}{2(2\nu-m+1)}\sqrt{\frac{2(\nu-m)-1}{\pi}}(1-\rho^2)^{2\nu-m-1}$ . \\
If $k\neq m$ then
\begin{equation*}
\mathcal L_{\nu}(\phi^{\nu,m}_k)(z)=\frac{\pi\gamma^{\nu,m}_k e^{i(k-m)t}}{2(k-m)}\left(I_3+I_4\right)
\end{equation*}
where 
$$
I_3=\frac{(1+k-m)_m}{m!(k-m+1)}\rho^{k-m+2}(1-\rho^2)^{2\nu-m-1}\ _2F_1\left(\begin{array}{cc}-m+1,2(\nu-m)+k\\ 2+k-m\end{array}\mid \rho^2\right)
$$
and 
$$
I_4=\frac{\alpha^{\nu,m}_k}{2\nu-m-1}(1-\rho^2)^{2\nu-m-1}\ _2F_1\left(\begin{array}{cc}-m+1,2\nu-m-1\\ 2(\nu-m),\end{array}\mid \rho^2\right)
$$
\end{proposition}
\begin{proof}[\textbf{Proof}]
For $k=m$, we have 

$$\mathcal{L}_{\nu}\left(\phi^{\nu,m}_k\right)(z)=\frac{(-1)^{m}}{\pi}\sqrt{\frac{2(\nu-m)-1}{\pi}}$$
$$\int_{\mathbb{D}}{\left(1-|\xi|^2\right)^{2\nu-m-2}P^{(0,2(\nu-m)-1)}_m\left(1-2|\xi|^2\right)\log\left(\left|z-\xi\right|\right)d\mu(\xi)}$$

$$=(-1)^{m}\sqrt{\frac{2(\nu-m)-1}{\pi}}\int_{0}^{1}{(1-r^2)^{2\nu-m-2}P^{(0,2(\nu-m)-1)}_m\left(1-2r^2\right)\log(\rho\wedge r)dr^2}$$
$$=\frac{(-1)^{m}}{2}\sqrt{\frac{2(\nu-m)-1}{\pi}}\int_{0}^{1}{(1-t)^{2\nu-m-2}P^{(0,2(\nu-m)-1)}_m\left(1-2t\right)\log(\rho^2\vee t)dt}$$
$$=\frac{(-1)^{m}}{2}\sqrt{\frac{2(\nu-m)-1}{\pi}}\left[I_1+I_2\right]
$$
Where 
$$
I_1=\int_{0}^{\rho^2}{(1-t)^{2\nu-m-2}P^{(0,2(\nu-m)-1)}_m\left(1-2t\right)\log(\rho^2\vee t)dt}
$$
and
$$
I_2=\int_{\rho^2}^{1}{(1-t)^{2\nu-m-2}P^{(0,2(\nu-m)-1)}_m\left(1-2t\right)\log(t)dt}
$$
Calculus of $I_1$.
$$
I_1=\log(\rho^2)\int_{\rho^2}^{1}{(1-t)^{2\nu-m-2}P^{(0,2(\nu-m)-1)}_m\left(1-2t\right)dt}
$$
We use the formula 
$$
P^{(\alpha,\beta)}_k(u)=\frac{(1+\alpha)_k}{k!}\ _2F_1\left(\begin{array}{cc}-k,1+\alpha+\beta+k\\ 1+\alpha\end{array}\mid\frac{1-u}{2}\right) 
$$
We have 
$$
I_1=\log(\rho^2)\int_{0}^{\rho^2}{(1-t)^{2\nu-m-2}\ _2F_1\left(\begin{array}{cc}-m,2\nu-m\\ 1\end{array}\mid t\right) dt}
$$
By \cite{PBM}, we have 
$$\int{x^{c-1}(1-x)^{b-c-1}\ _2F_1\left(\begin{array}{cc}a,b\\ c\end{array}\mid x\right) dx}=\frac{1}{c}x^{c}(1-x)^{b-c}\ _2F_1\left(\begin{array}{cc} a+1,b\\ c+1\end{array}\mid x\right)
$$
implies that
$$I_1=\log(\rho^2)\rho^2(1-\rho^2)^{2\nu-m-1}\ _2F_1\left(\begin{array}{cc} -m+1,2\nu-m\\ 2\end{array}\mid \rho^2\right)$$
Calculus of $I_2$.
$$I_2=\int_{\rho^2}^{1}{(1-t)^{2\nu-m-2}P^{(0,2(\nu-m)-1)}_m\left(1-2t\right)\log(t)dt}$$
Use the previous formula in \cite{PBM} and the integration by part gives
$$I_2=\left[t(1-t)^{2\nu-m-1}\ _2F_1\left(\begin{array}{cc} -m+1,2\nu-m\\ 2\end{array}\mid t\right)\log(t)\right]^{1}_{\rho^2}$$
$$-\int_{\rho^2}^{1}{(1-t)^{2\nu-m}\ _2F_1\left(\begin{array}{cc} -m+1,2\nu-m\\ 2\end{array}\mid t\right) dt}$$

$$=-\rho^2\log(\rho^2)(1-\rho^2)^{2\nu-m-1}\ _2F_1\left(\begin{array}{cc} -m+1,2\nu-m\\ 2\end{array}\mid \rho^2\right)$$
$$-\int_{\rho^2}^{1}{(1-t)^{2\nu-m}\ _2F_1\left(\begin{array}{cc} -m+1,2\nu-m\\ 2\end{array}\mid t\right) dt}$$
Calculus of
$$
\int_{\rho^2}^{1}{(1-t)^{2\nu-m}\ _2F_1\left(\begin{array}{cc} -m+1,2\nu-m\\ 2\end{array}\mid t\right) dt}
$$
Use the following formula which has place in \cite{SM}
$$\ _2F_1\left(\begin{array}{cc} a,b\\ c\end{array}\mid t\right)=\frac{\Gamma(c)\Gamma(c-a-b)}{\Gamma(c-a)\Gamma(c-b)}\ _2F_1\left(\begin{array}{cc} a,b\\ a+b-c+1\end{array}\mid 1-t\right)$$
$$+\frac{\Gamma(c)\Gamma(a+b-c)}{\Gamma(a)\Gamma(b)}(1-t)^{c-a-b}\ _2F_1\left(\begin{array}{cc} a,b\\ a+b-c+1\end{array}\mid 1-t\right)$$
We put $a=1-m$, $b=2\nu-m$, $c=2$ and use the formula Boher-Mollerup, for $z\in\mathbb{R}^*_+$, 
$$\Gamma(z)=\frac{e^{-\gamma z}}{z}\prod_{n=1}^{\infty}{\left(1+\frac{z}{n}\right)^{-1}e^{-\frac{z}{n}}}$$
which implies $\frac{1}{\Gamma(1-m)}=0$, then 
$$\ _2F_1\left(\begin{array}{cc} -m+1,2\nu-m\\ 2\end{array}\mid t\right)=\frac{2\Gamma(2(m-\nu)+1)}{m!\Gamma(2+m-2\nu)}\ _2F_1\left(\begin{array}{cc} -m+1,2\nu-m\\ 2(\nu-m)\end{array}\mid 1-t\right)$$
implies that
$$\int_{\rho^2}^{1}{(1-t)^{2\nu-m}\ _2F_1\left(\begin{array}{cc} -m+1,2\nu-m\\ 2\end{array}\mid t\right) dt}$$
$$=\frac{2\Gamma(2(m-\nu)+1)}{m!\Gamma(2+m-2\nu)}\int_{\rho^2}^{1}{(1-t)^{2\nu-m}\ _2F_1\left(\begin{array}{cc} -m+1,2\nu-m\\ 2(\nu-m)\end{array}\mid 1-t\right) dt}$$
By the change $1-t=s$, we get
$$\int_{\rho^2}^{1}{(1-t)^{2\nu-m}\ _2F_1\left(\begin{array}{cc} -m+1,2\nu-m\\ 2\end{array}\mid t\right) dt}$$
$$=\frac{2\Gamma(2(m-\nu)+1)}{m!\Gamma(2+m-2\nu)}\int_{0}^{1-\rho^2}{t^{2\nu-m}\ _2F_1\left(\begin{array}{cc} -m+1,2\nu-m\\ 2(\nu-m)\end{array}\mid t\right)dt}$$
In \cite{PBM} page $44$, 
$$\int{x^{\alpha-1}\ _2F_1\left(\begin{array}{cc} a,b\\ c\end{array}\mid -t\right)dx}=\frac{x^{\alpha}}{\alpha}\ _3F_2\left(\begin{array}{cc} a,b,\alpha\\ c,\alpha+1\end{array}\mid -t\right)+\frac{\Gamma(\alpha)\Gamma(a-\alpha)\Gamma(b-\alpha)\Gamma(c)}{\Gamma(a)\Gamma(b)\Gamma(c-\alpha)}$$ 
Since $a=1-m$, $b=2\nu-m$, $c=2(\nu-m)$, and $\alpha=2\nu-m+1$ we have
$$\frac{\Gamma(\alpha)\Gamma(a-\alpha)\Gamma(b-\alpha)\Gamma(c)}{\Gamma(a)\Gamma(b)\Gamma(c-\alpha)}=0$$
and by the change $t=-s$ 
$$\int_{0}^{1-\rho^2}{t^{2\nu-m+1}\ _2F_1\left(\begin{array}{cc} -m+1,2\nu-m\\ 2(\nu-m)\end{array}\mid t\right) dt}$$
$$=(-1)^{m}\int_{0}^{\rho^2}{t^{2\nu-m}\ _2F_1\left(\begin{array}{cc} -m+1,2\nu-m\\ 2(\nu-m)\end{array}\mid -t\right) dt}$$
$$=(-1)^m\frac{(\rho^2-1)^{2\nu-m+1}}{2\nu-m+1}\ _3F_2\left(\begin{array}{cc} -m+1,2\nu-m,2\nu-m+1\\ 2(\nu-m),2\nu-m+2\end{array}\mid 1-\rho^2\right)$$
we set $\alpha^{\nu,m}_k=\frac{2\Gamma(2(m-\nu)+1)}{m!\Gamma(2+m-2\nu)}$. We get
$$I_2=-\rho^2\log(\rho^2)(1-\rho^2)^{2\nu-m-1}\ _2F_1\left(\begin{array}{cc} -m+1,2\nu-m\\ 2\end{array}\mid \rho^2\right)$$
$$+(-1)^{m}\alpha^{\nu,m}_k\frac{(1-\rho^2)^{2\nu-m-1}}{2\nu-m+1}\ _3F_2\left(\begin{array}{cc} -m+1,2\nu-m,2\nu-m+1\\ 2(\nu-m), 2\nu-m+2\end{array}\mid 1-\rho^2\right)$$
Finally $$\mathcal L_{\nu}(\phi^{\nu,m}_k)(z)=\frac{\alpha^{\nu,m}_k}{2(2\nu-m+1)}\sqrt{\frac{2(\nu-m)-1}{\pi}}(1-\rho^2)^{2\nu-m-1}$$
$$\times \ _3F_2\left(\begin{array}{cc} -m+1,2\nu-m,2\nu-m+1\\ 2(\nu-m), 2\nu-m+2\end{array}\mid 1-\rho^2\right)$$ 
Now if $k>m$, set $z=\rho e^{it}$.

$$\mathcal L_{\nu}(\phi^{\nu,m}_k)(z)=\gamma^{\nu,m}_k\int_{\mathbb{D}}{(1-|\xi|^2)^{2\nu-m-2}\xi^{k-m}\log\left(\frac{1}{|z-\xi|}\right)P^{(k-m,2(\nu-m)-1)}_m\left(1-2|\xi|^2\right) d\mu(\xi)}$$

$$=\gamma^{\nu,m}_k\int_{0}^{1}{(1-r^2)^{2\nu-m-2}r^{k-m+1}P^{(k-m,2(\nu-m)-1)}_m\left(1-2r^2\right)\int_{0}^{2\pi}{e^{i(k-m)\theta}\log\left(\frac{1}{|z-r^{i\theta}|}\right)d\theta}dr}$$
$$=\frac{\pi\gamma^{\nu,m}_k e^{i(k-m)t}}{2(k-m)}\int_{0}^{1}{(1-r^2)^{2\nu-m-2}r^{k-m}P^{(k-m,2(\nu-m)-1)}_m\left(1-2r^2\right)\left((\frac{r}{\rho})^{k-m}\wedge (\frac{\rho}{r})^{k-m}\right) dr^2}$$
$$=\frac{\pi\gamma^{\nu,m}_k e^{i(k-m)t}}{2(k-m)} \Big( \int_{0}^{\rho}{(1-r^2)^{2\nu-m-2}r^{k-m}P^{(k-m,2(\nu-m)-1)}_m
\left(1-2r^2\right)\left((\frac{r}{\rho})^{k-m}\wedge (\frac{\rho}{r})^{k-m}\right) dr^2} $$
$$+\int_{\rho}^{1}{(1-r^2)^{2\nu-m-2}r^{k-m}P^{(k-m,2(\nu-m)-1)}_m\left(1-2r^2\right) 
\left((\frac{r}{\rho})^{k-m}\wedge (\frac{\rho}{r})^{k-m}\right) dr^2} \Big)$$
We set $$I_3=\int_{0}^{\rho}{(1-r^2)^{2\nu-m-2}r^{k-m}P^{(k-m,2(\nu-m)-1)}_m\left(1-2r^2\right)\left((\frac{r}{\rho})^{k-m}\wedge (\frac{\rho}{r})^{k-m}\right) dr^2}$$
and $$I_4=\int_{\rho}^{1}{(1-r^2)^{2\nu-m-2}r^{k-m}P^{(k-m,2(\nu-m)-1)}_m\left(1-2r^2\right)\left((\frac{r}{\rho})^{k-m}\wedge (\frac{\rho}{r})^{k-m}\right) dr^2}$$
Calculus of $I_3$.
$$I_3=\frac{\rho^{m-k}(1+k-m)_m}{m!}\int_{0}^{\rho^2}{t^{k-m}(1-t)^{2\nu-m-2}\ _2F_1\left(\begin{array}{cc}-m,2(\nu-m)+k\\ 1+k-m\end{array}\mid t\right) dt}$$
By the formula $$\int{x^{c-1}(1-x)^{b-c-1}\ _2F_1\left(\begin{array}{cc}a,b\\ c\end{array}\mid x\right) dx}=\frac{1}{c}x^{c}(1-x)^{b-c}\ _2F_1\left(\begin{array}{cc} a+1,b\\ c+1\end{array}\mid x\right)$$
we have
$$I_3=\frac{(1+k-m)_m}{m!(k-m+1)}\rho^{k-m+2}(1-\rho^2)^{2\nu-m-1}\ _2F_1\left(\begin{array}{cc}-m+1,2(\nu-m)+k\\ 2+k-m\end{array}\mid \rho^2\right)$$
Calculus of $I_4$.
$$I_4=\int_{\rho}^{1}{(1-r^2)^{2\nu-m-2}r^{k-m}P^{(k-m,2(\nu-m)-1)}_m\left(1-2r^2\right)\left((\frac{r}{\rho})^{k-m}\wedge (\frac{\rho}{r})^{k-m}\right) dr^2}$$
$$=\frac{\rho^{k-m}(1+k-m)_m}{2m!}\int_{\rho^2}^{1}{(1-t)^{2\nu-m-2}\ _2F_1\left(\begin{array}{cc}-m,2(\nu-m)+k\\ 1+k-m\end{array}\mid t\right) dt}$$
As the previous 
$$\int_{\rho^2}^{1}{(1-t)^{2\nu-m-2}\ _2F_1\left(\begin{array}{cc}-m,2(\nu-m)+k\\ 1+k-m\end{array}\mid t\right) dt}$$
$$=\alpha^{\nu,m}_k \int_{\rho^2}^{1}{(1-t)^{2\nu-m-2}\ _2F_1\left(\begin{array}{cc}-m+1,2\nu-m\\ 2(\nu-m)\end{array}\mid 1-t\right) dt}$$
$$=(-1)^m\alpha^{\nu,m}_k \int_{\rho^2-1}^{0}{t^{2\nu-m-2}\ _2F_1\left(\begin{array}{cc}-m+1,2\nu-m\\ 2(\nu-m)\end{array}\mid t\right) dt}$$
$$=\frac{\alpha^{\nu,m}_k}{2\nu-m-1}(1-\rho^2)^{2\nu-m-1}\ _3F_2\left(\begin{array}{cc}-m+1,2\nu-m,2\nu-m-1\\ 2(\nu-m),2\nu-m\end{array}\mid 1-\rho^2\right)$$
also $$\ _3F_2\left(\begin{array}{cc}-m+1,2\nu-m,2\nu-m-1\\ 2(\nu-m),2\nu-m\end{array}\mid 1-\rho^2\right)=\ _2F_1\left(\begin{array}{cc}-m+1,2\nu-m-1\\ 2(\nu-m)\end{array}\mid 1-\rho^2\right)$$
Now if $k<m$. We have $$\phi^{\nu,m}_k(z)=(-1)^k\sqrt{\frac{2(\nu-m)-1}{\pi}\frac{k!\Gamma(2(\nu-m)+m)}{m!\Gamma(2(\nu-m)+k)}}(1-|z|^2)^{-m}\overline{z}^{m-k}P^{(m-k,2(\nu-m)-1)}_k\left(1-2|z|^2\right)$$
By the formula 
\begin{equation*}
\frac{\Gamma \left( m+1\right) }{\Gamma \left( m-s+1\right) }P_{m}^{\left(
-s,\alpha \right) }\left( u\right) =\frac{\Gamma \left( m+\alpha +1\right) }{%
\Gamma \left( m-s+\alpha +1\right) }\left( \frac{u-1}{2}\right)
^{s}P_{m-s}^{\left( s,\alpha \right) }\left( u\right) ,1\leq s\leq m
\end{equation*}\label{s}%
and put $s=m-k$ and $\alpha=2(\nu-m)-1$, we have
$$P^{(m-k,2(\nu-m)-1)}_k\left(1-2|z|^2\right)=\frac{m!\Gamma(k+\alpha+1)}{k!\Gamma(m+\alpha+1)}P^{(k-m,2(\nu-m)-1)}_m\left(1-2|z|^2\right)$$
substituting in the expression of $\phi^{\nu,m}_k(z)$, we get
$$\phi^{\nu,m}_k(z)=(-1)^m\sqrt{\frac{2(\nu-m)-1}{\pi}\frac{m!\Gamma(2(\nu-m)+k)}{k!\Gamma(2(\nu-m)+m)}}(1-|z|^2)^{-m}z^{k-m}P^{(k-m,2(\nu-m)-1)}_m\left(1-2|z|^2\right)$$
it's the same formula for $k>m$, which prove the same formula of $\mathcal L_{\nu}(\phi^{\nu,m}_k)(z)$ if $k>m$.
\end{proof}
\begin{remark}
By the previous formula in \cite{SM}, we have
$$\ _2F_1\left(\begin{array}{cc}-m+1,2(\nu-m)+k\\ 2(\nu-m)\end{array}\mid \rho^2\right)=\frac{k!\Gamma(2+k-m)}{\Gamma(1-2(\nu-m))}\ _2F_1\left(\begin{array}{cc}-m+1,2(\nu-m)+k\\ 2(\nu-m)\end{array}\mid 1-\rho^2\right)$$
\end{remark}

\subsection{The spectrum of $\mathcal L_{\nu}$}
\begin{proposition}
If $k\neq m$, then
 $$\lambda_k=\sqrt{J_1+J_2+J_3}$$
where
$$J_1=\left(\frac{(1+k-m)_m}{m!(k-m+1)}\right)^2\sum_{n=0}^{\infty}{A_n\frac{\Gamma(2n+2k-2m+6-1)\Gamma(4\nu-2m-1)}{\Gamma(2n+2k-4m+4\nu+6)}}$$

$$J_2=\left(\frac{\alpha^{\nu,m}_k}{2\nu-m-1}\right)^2\sum_{n=0}^{\infty}{A_n\frac{\Gamma(4\nu-2m-1)\Gamma(2n+2)}{\Gamma(2n+4\nu-2m+1)}}$$

and 

$$ J_3=\frac{(1+k-m)_m \alpha^{\nu,m}_k}{m!(k-m+1)(2\nu-m-1)}\left(\sum_{n=0}^{\infty}{A_n\frac{\Gamma(k-m+2)\Gamma(4\nu-2m-1)}{\Gamma(4\nu-k-3m)}}\right)$$

If $k=m$ then
\begin{equation}
\lambda^2_k=\frac{\alpha^{\nu,m}_k \left(2(2\nu-m)-1\right)}{8\left(\pi(2\nu-m+1)\right)}\sum_{n=0}^{\infty}{\frac{B_n}{n+2\nu-m}}
\end{equation}
where 
$$B_n=\sum_{n=0}^{\infty}{\frac{\Gamma(-m+1)\Gamma(2\nu-m)\Gamma(2(\nu-m)+1)}{n!\Gamma(2(\nu-m)\Gamma(2\nu-m+2)} }$$
\end{proposition}
\begin{proof}[\textbf{Proof}]

If $k\neq m$. We have $$\left(\mathcal L_{\nu}(\phi^{\nu,m}_k)\right)(z)=\frac{\pi\gamma^{\nu,m}_k\left(I_3+I_4\right)}{2(k-m)}e^{i(k-m)t}$$

We set  $\mathcal H=\left(L^2(\mathbb{D}),(1-|\xi|^2)^{2\nu-2}d\mu(\xi) \right)$,  $I_3=I_3(\rho)$, and $I_4=I_4(\rho)$  we have

$$\lambda^2_k=\left\langle \mathcal L_{\nu}(\phi^{\nu,m}_k),\mathcal L_{\nu}(\phi^{\nu,m}_k)\right\rangle_{\mathcal H}$$

$$=\frac{\pi^2\gamma^{\nu,m}_k}{(k-m)}\int_{0}^{1}{\left(I_3(\rho)+I_4(\rho)\right)^2 \rho d\rho}$$

Calculus of $\int_{0}^{1}{\left(I_3(\rho)\right)^2 \rho d\rho}$.

$$I_3(\rho)=\frac{(1+k-m)_m}{m!(k-m+1)}\rho^{k-m+2}(1-\rho^2)^{2\nu-m-1}\ _2F_1\left(\begin{array}{cc}-m+1,2(\nu-m)+k\\ 2+k-m\end{array}\mid \rho^2\right)$$

Since

 $$\ _2F_1\left(\begin{array}{cc}-m+1,2(\nu-m)+k\\ 2+k-m\end{array}\mid \rho^2\right)=\sum_{n=0}^{\infty}{\frac{(-m+1)_n(2(\nu-m)+k)_n}{(2+k-m)_n}\frac{\rho^{2n}}{n!}}$$

then

$$\left(I_3(\rho)\right)^2=\left(\frac{(1+k-m)_m}{m!(k-m+1)}\right)^2\sum_{n=0}^{\infty}{A_n \rho^{2n}(1-\rho^2)^{4\nu-2m-2}}$$

where

 $$A_n=\frac{1}{n!}\sum_{i=0}^{n}{\frac{(-m+1)_i(-m+1)_{n-i} (2(\nu-m)+k)_i(2(\nu-m)+k)_{n-i}}{(2(\nu-m))_i(2(\nu-m))_{n-i}}}$$

Thus

$$J_1=\int_{0}^{1}{\left(I_3(\rho)\right)^2 \rho d\rho}=\left(\frac{(1+k-m)_m}{m!(k-m+1)}\right)^2\sum_{n=0}^{\infty}{A_n\int_{0}^{1}{\rho^{2n+2k-2m+6-1}(1-\rho^2)^{4\nu-2m-1-1} d\rho}}$$

Use the fact that 

$$\int_{0}^{1}{t^{\alpha-1}(1-t)^{\beta-1} dt}=\frac{\Gamma(\alpha)\Gamma(\beta)}{\Gamma(\alpha+\beta)}$$

implies
\begin{equation}
\int_{0}^{1}{\left(I_3(\rho)\right)^2 \rho d\rho}=\left(\frac{(1+k-m)_m}{m!(k-m+1)}\right)^2
\sum_{n=0}^{\infty}{A_n\frac{\Gamma(2n+2k-2m+6-1)\Gamma(4\nu-2m-1)}{\Gamma(2n+2k-4m+4\nu+6)}}
\end{equation}
Calculus of $\int_{0}^{1}{\left(I_4(\rho)\right)^2 \rho d\rho}$.

In the same
\begin{equation}
J_2=\int_{0}^{1}{\left(I_4(\rho)\right)^2 \rho d\rho}=\left(\frac{\alpha^{\nu,m}_k}{2\nu-m-1}\right)^2\sum_{n=0}^{\infty}{A_n\frac{\Gamma(4
\nu-2m-1)\Gamma(2n+2)}{\Gamma(2n+4\nu-2m+1)}}
\end{equation}

Calculus of $2\int_{0}^{1}{\left(I_3(\rho)\right)\left(I_4(\rho)\right) \rho d\rho}$.
\begin{equation}
J_3=2\int_{0}^{1}{\left(I_3(\rho)\right)\left(I_4(\rho)\right) \rho d\rho}=\frac{(1+k-m)_m \alpha^{\nu,m}_k}{m!(k-m+1)(2\nu-m-1)}
\sum_{n=0}^{\infty}{A_n\frac{\Gamma(k-m+2)\Gamma(4\nu-2m-1)}{\Gamma(4\nu-k-3m)}}
\end{equation}

If $k=m$.\\

Since 

$$\left(\ _3F_2\left(\begin{array}{cc}-m+1,2\nu-m,2(\nu-m)+1\\ 2(\nu-m),2\nu-m+2\end{array}\mid 1-\rho^2\right)\right)^2$$
$$=\sum_{n=0}^{\infty}{\frac{\Gamma(-m+1)\Gamma(2\nu-m)\Gamma(2(\nu-m)+1)}{n!\Gamma(2(\nu-m)\Gamma(2\nu-m+2)} (1-\rho^2)^{n}}$$

$$\lambda^2_k=\frac{\alpha^{\nu,m}_k \left(2(2\nu-m)-1\right)}{8\left(\pi(2\nu-m+1)\right)}\sum_{n=0}^{\infty}{B_n\int_{0}^{1}{(1-\rho^2)^{n+2\nu-m-1}d\rho}}$$
$$=\frac{\alpha^{\nu,m}_k \left(2(2\nu-m)-1\right)}{8\left(\pi(2\nu-m+1)\right)}\sum_{n=0}^{\infty}{\frac{B_n}{n+2\nu-m}}$$
where 
$$B_n=\sum_{n=0}^{\infty}{\frac{\Gamma(-m+1)\Gamma(2\nu-m)\Gamma(2(\nu-m)+1)}{n!\Gamma(2(\nu-m)\Gamma(2\nu-m+2)} }$$

\end{proof}

\section{Asymptotic behavior of singular values $\lambda_k$ as $k\longrightarrow\infty$}

\begin{proposition}
$$\lambda_k\sim C \sqrt{k^{m-4\nu+1}},\ \text{as}\ k\longrightarrow\infty$$
where $C$ is a constant
\end{proposition}
\begin{proof}[\textbf{Proof}]
If $k> m$, then
 $$\lambda_k=\sqrt{J_1+J_2+J_3}$$
where
$$J_1=\left(\frac{(1+k-m)_m}{m!(k-m+1)}\right)^2\sum_{n=0}^{\infty}{A_n\frac{\Gamma(2n+2k-2m+6-1)\Gamma(4\nu-2m-1)}{\Gamma(2n+2k-4m+4\nu+6)}}$$

$$J_2=\left(\frac{\alpha^{\nu,m}_k}{2\nu-m-1}\right)^2\sum_{n=0}^{\infty}{A_n\frac{\Gamma(4\nu-2m-1)\Gamma(2n+2)}{\Gamma(2n+4\nu-2m+1)}}$$

and 

$$ J_3=\frac{(1+k-m)_m \alpha^{\nu,m}_k}{m!(k-m+1)(2\nu-m-1)}\left(\sum_{n=0}^{\infty}{A_n\frac{\Gamma(k-m+2)\Gamma(4\nu-2m-1)}{\Gamma(4\nu-k-3m)}}\right)$$

The limit of $\lambda_k$ as $k\longrightarrow\infty$.\\ 
We use the formula $$\frac{\Gamma(k+a)}{\Gamma(k+b)}\sim k^{a-b}$$
we have 
\begin{equation}
J_1 \sim \left(\frac{k^{-1-m}}{m!}\right)^2\sum_{n=0}^{\infty}{A_n\Gamma(4\nu-2m-1) (2k)^{2m-4\nu-1}}
\sim k^{-4\nu-1} 2^{2m-4\nu-1}\Gamma(4\nu-2m-1)\sum_{n=0}^{\infty}{\frac{A_n}{m!}}
\end{equation}
\begin{equation}
J_2=\mathcal O_{k\sim\infty} \left(1\right)
\end{equation}
In the same
\begin{equation}
J_3\sim k^{m-4\nu+1}\frac{\alpha^{\nu,m}_k \Gamma(4\nu-2m-1)}{m!(2\nu-m-1)}\sum_{n=0}^{\infty}{A_n}
\end{equation}
Therefore 
$$\lambda_k\sim C \sqrt{k^{m-4\nu+1}}$$
where $C$ is a constant
\end{proof}

\providecommand{\href}[2]{#2}

\address{
Laboratoire de Math\'ematiques Appliqu\'ees \& Calcul Scientifique, \\ 
Sultan Moulay Slimane University, BP 523, 23000  Beni Mellal, Morocco\\ 
\email{mhamedmaster@gmail.com}\\
}


\begin{thebibliography}{10}

\bibitem[1]{AKL1990}: Anderson, J.M.; Khavinson, D. and Lomonosov, V.Spectral properties of some operators arising in potential theory, 1990,
preprint, 29pp.\\
\bibitem[2]{AK}: J.Arazy , D.Khavinson spectral estimates of Cauchy's transform in $L^{2}(\Omega )$. Integr Equat Oper Th. Vol. 15 (1992).\\
\bibitem[3]{B}: G. I. Barenblatt Dimensional Analysis, (Gordon Breach Science Publishers, New York, 1987).\\ 
\bibitem[4]{D}: D. Jana Phys. Edu.,25, 35 (2008).\\
\bibitem[5]{D2002}:  D. Jana Phys. Edu.,19, 167 (2002).\\
\bibitem[6]{L}: H. L. Langhaar. Dimensional Analysis and theory of models, (Wiley, New York,1951).\\
\bibitem[7]{M}: Multin R. Dostanic. The properties of Cauchy transform on a bounded domain. J.operator theory 36(1996 ), 233-247.\\
\bibitem[8]{Mu}: Multin R. Dostanic. Norm estimate of the Cauchy transform on $L^{p}(\Omega )$.\\
\bibitem[9]{MO}:Mourad E.H.Ismail, Classical and Quantum Orthogonal Polynomials in one variable, Encyclopedia of Mathematics and its
applications, Cambridge university press(2005).\\ 
\bibitem[10]{FV}: Ferapontov, E.V., Veselov, A.P. Integrable Schrodinger operators with magnetic fields: factorization method on curved surfaces. J. Math. Phys. 42, 590-607 (2001).\\
\bibitem[11]{PBM}: A.P.Prudnikov , Yu.A.Brychkov ,O.I.Marichev . Integrals and Series Volume 3 More spacial functions Gordon and Breach , New York , 1990.\\
\bibitem[12]{S}: A. Szeg\"{o}, Orthogonal polynomials. American Mathematical Society; Providence, R.I. (1975).\\
\bibitem[13]{SM}:H. Srivastava and L. Manocha, A Treatise on Generating Functions, Ellis Horwood Ltd, London 1984.\\
\bibitem[14]{W}:W.Magnus, F.Oberhettinger and R.P.Soni, Formulas and Theorems for the Special Functions of Mathematical Physics, Springer-Verlag Berlin Heidelberg New york, 1966.
\end{thebibliography}
\end{document}